\documentclass[11pt,letterpaper,reqno]{amsart}
\usepackage[reqno]{amsmath}
\usepackage[utf8]{inputenc}
\usepackage[T1]{fontenc}
\usepackage{enumitem}
\usepackage{amsmath}
\usepackage{amsthm}
\usepackage{amsfonts}
\usepackage{amssymb}
\usepackage{graphicx}
\usepackage[mathscr]{euscript}
\DeclareSymbolFont{rsfs}{U}{rsfs}{m}{n}
\DeclareSymbolFontAlphabet{\mathscrsfs}{rsfs}
\newcounter{atheorem}
\renewcommand{\theatheorem}{\Alph{atheorem}}

\newenvironment{atheorem}
  {\refstepcounter{atheorem}
   \par\medskip
   \noindent\textbf{Theorem \theatheorem.\, }\itshape}
  {\par\normalfont\medskip}
\newtheorem{ntheorem}{Theorem}[section]

\newtheorem{lemma}[ntheorem]{Lemma}

\usepackage{blindtext}
\usepackage[english]{babel}
\providecommand{\subjclass}[2][]{\textbf{Mathematics Subject Classification #1:} #2}
\usepackage{hyperref} 
\usepackage{titlesec}
\titlelabel{\thetitle.\quad}
\title{Sections and Chapters}

\counterwithin{equation}{section}

\usepackage[left=3cm,right=3cm,top=3cm,bottom=3cm]{geometry}
\title{ A normality criterion for a family of meromorphic functions}
\date{}
\author{Kuntal Mandal}
\address{Department of Mathematics, Visva-Bharati, Santiniketan, West Bengal, India 731235.}
\email{mkuntal1994@gmail.com}
\author{Bipul Pal}
\address{Department of Mathematics, Visva-Bharati, Santiniketan, West Bengal, India 731235.}
\email{palbipul86@gmail.com, bipul.pal@visva-bharati.ac.in}
\thanks {The work of the first author was supported by UGC (UGC fellowship for PhD (NTA Ref. No.-211610061981)), India.}
    

\begin{document}
\begin{abstract}
We consider a family  $\mathscr{F}$ of meromorphic functions defined in a domain $D$, a holomorphic function $\psi$ and a homogeneous differential polynomial $ P[f] $ of degree $d$ with weight $w$. In this paper, we prove the normality of $\mathscr{F}$ under certain conditions such as $f\neq 0$, $P[f]\neq 0$ and all the zeros of the function $P[f] - \psi^d$ have multipicity at least  $\displaystyle{\frac{w+1}{w-1}}$, for each $f \in \mathscr{F}$.
	
\end{abstract}

    \maketitle
    
	\section{\textbf{Introduction}}
	  
Let $D$ be a domain in $\mathbb{C}$ and $\mathscr{F}$ be a family of meromorphic functions defined on $D$. Then $\mathscr{F}$ is said to be normal on $D$,  in the sense of Montel, if for any sequence $\{f_n\}\subset \mathscr{F}$ there exists a subsequence $\{f_{n_j}\}$, which converges spherically locally uniformly in $D$, to a meromorphic function or to $\infty$ (see \cite{hayman1964meromorphic,schiff1993normal,Lo1993value}). The following is Gu's normality criterion obtained in 1979 (see \cite{gu1979normal}, one of the well-known findings on normal families of meromorphic functions.

\begin{atheorem}\label{a}
  Let $\mathscr{F}$ be a family of meromorphic functions defined in $D$ and let $k$ be a positive integer. If, for every function $f\in \mathscr{F}$, $f \neq 0$, $f^{(k)} \neq 1$, then $\mathscr{F}$ is normal in $D$.
\end{atheorem}
Several extensions of this result have been developed. This is still true, as proved by some authors, if $f^{(k)} \neq 1$ is weakened to "the zeros of $f^{(k)}- 1$ have multiplicities at least $l$, where $l\ge k+5+[2/k]$" (see \cite{chen1994normality, fang1994criteria, Lo1993value}, etc.). As demonstrated by Fang and Chang \cite{fang2007normal} in 2007, if $f^{(k)} \neq 0$, the lower bound of $l$ can be substantially less, as follows.
\begin{atheorem}\label{b}
    Let $\mathscr{F}$ be a family of meromorphic functions defined in D and let $k$ be a positive integer. If, for every function $f \in \mathscr{F}$, $f \neq 0, f^{(k)} \neq 0$, and all zeros of $f^{(k)}-1$ have multiplicity at least $\displaystyle{\frac{k + 2}{k}}$, then $\mathscr{F}$ is normal in $D$.
\end{atheorem}
Jiying Xia and Yan Xu \cite{xia2009normal} established the following generalisation in 2009, in response to the obvious question of how the outcome in Theorem \ref{b} might be extended when the constant $1$ is replaced by a small function $\psi(z)\not\equiv0$.
\begin{atheorem}\label{c}
Let $\mathscr{F}$ be a family of meromorphic functions defined in a domain $D \subset \mathbb{C}$, let $\psi (\not\equiv 0)$ and $a_0, a_1, \ldots, a_{k-1}$ be holomorphic functions in $D$, and let $k$ be a positive integer. Suppose that, for every function $f \in F$, $f \neq 0$, $f^{(k)} + a_{k-1} f^{(k-1)} + \cdots + a_1 f' + a_0 f \neq 0,$ and all zeros of $f^{(k)}(z) + a_{k-1}(z) f^{(k-1)}(z) + \cdots + a_1(z) f'(z) + a_0(z) f(z) - \psi(z)$ have multiplicity at least $\displaystyle{\frac{k+2}{k}}$. If, for $k = 1$, $\psi$ has only zeros of multiplicity at most $2$ and for $k \ge 2$, $\psi$ has only simple zeros, then $\mathscr{F}$ is a normal family in $D$.
\end{atheorem}
 Earlier results by Schwick \cite{schwick1997exceptional} in $1997$ and Yang \cite{yang1986normal} in $1986$ had already shown that the conclusion of the Theorem \ref{a} still holds even when $1$ is replaced by a non-zero holomorphic function  $\psi (\not\equiv 0)$.

 A natural question arises: what can be said if the linear differential polynomial
 $f^{(k)}(z) + a_{k-1}(z) f^{(k-1)}(z) + \cdots + a_1(z) f'(z) + a_0(z) f(z)$ in Theorem \ref{c} is replaced by a \textit{non-linear differential polynomial} in $f$? 

To investigate this idea, we first introduce some basic definitions that will guide the proof of our main result.

Consider the differential polynomial
	\begin{equation}\label{M1}
		P[f]=P[f,f',\cdots,f^{(k)}]=\sum_{j=1}^{n}a_jM_j[f],
	\end{equation}
	where  each monomial term is
	\begin{equation*}
		M_j[f]=M_j[f,f',\cdots,f^{(k)}]=\prod_{i=0}^{k}\left(f^{(i )}\right)^{n_{i, j}}
	\end{equation*}
	and each $a_j$ is a holomorphic function  in $D$.

    The degree and weight of $M_j[f]$ are defined by $\displaystyle d_j=d(M_j[f])=\sum_{i=0}^{k}n_{i, j}$ and $\displaystyle w_j=w(M_j[f])=\sum_{i=0}^{k}(1+i)n_{i, j}$ respectively. 
    For $P[f]$, the degree and weight are  $$\displaystyle{d=d(P)=\max_{1\le j\le n}d_j}\;\;  \text{and} \;\;\displaystyle{w=w(P)=\max_{1\le j\le n}w_j}$$ and its lower degree is $\displaystyle{\alpha =\underline{d}=\min_{1\le j\le n} d_j}$.
    
    If $\alpha=d$, then $P[f]$ is said to be a homogeneous differential polynomial.

    We now present our main result, which extends the above theorems to the setting of homogeneous differential polynomials, where the weight occurs for only one monomial and the corresponding $a_i$ is a zero-free holomorphic function.
    \begin{ntheorem}\label{m}
        Let $\mathscr{F}$ be a family of meromorphic functions defined in a domain $D \subset \mathbb{C}$, let $\psi (\not\equiv 0)$ and $a_1, a_2, \ldots, a_{n}$ be holomorphic functions in $D$, and let $k$, $n$ be two positive integers. Suppose $P[f]$ be a homogeneous differential polynomial of degree $d$ and weight $w$ such that  $P[f]$ has only one monomial $M_i[f]$ for some $i\in \{1,2,\cdots,n\}$ with $w$ as its weight, for which the corresponding coefficient $a_i$ is a zero-free holomorphic function. Assume that for every function $f \in \mathscr{F}$, $f \neq 0$, $P[f] \neq 0,$ and all zeros of $P[f] - \psi^d(z)$ have multiplicity at least $\displaystyle{\frac{w+1}{w-1}}$. If, for $w = 2$, $\psi$ has only zeros of multiplicity at most $2$ and for $w \ge 3$, $\psi$ has only simple zeros, then $\mathscr{F}$ is a normal in $D$. 
    \end{ntheorem}
    Clearly, Theorem \ref{m} generalizes Theorem \ref{a}, Theorem \ref{b} and Theorem \ref{c} and also bridges the gap between linear differential polynomials and broader classes of differential polynomials, shedding new light on the structure and behavior of normal families of meromorphic functions.
    \section{\textbf{Lemmas}}
    Among the many refinements of Zalcman’s lemma, a fundamental result in the theory of normal families, the following local version due to Pang and Zalcman \cite{pang2000normal} (cf.\cite{ChenGu1993an, huaihui1996yosida, yuefei1998picard,  zalcman1975heuristic, zalcman1998normal}) is of particular importance.
    \begin{lemma}\label{3l1}
        Let $k$ be a positive integer and let $\mathscr{F}$ be a family of functions meromorphic in a domain $D$ such that each function $f \in \mathscr{F}$ has only zeros of multiplicity at least $k$. Suppose that there exists a constant $A \ge 1$ such that $\displaystyle{|f^{(k)}(z)| \le A \text{ whenever } f(z) = 0, \; f \in \mathscr{F}.}$ If $\mathscr{F}$ is not normal at $z_0 \in D$, then for each $\alpha$ with $0 \le \alpha \le k$, there exist a sequence of points $z_j \in D$ with $z_j \to z_0$, a sequence of positive numbers $\rho_j \to 0$ and a sequence of functions $f_j \in \mathscr{F}$ such that $\displaystyle{g_j(\zeta) = \frac{f_j(z_j + \rho_j \zeta)}{\rho_j^{\alpha}} \longrightarrow g(\zeta)}$ locally uniformly with respect to the spherical metric, where $g$ is a non-constant meromorphic function on $\mathbb{C}$, all of whose zeros have multiplicity at least $k$ and such that $g^{\#}(\zeta) \le g^{\#}(0) = kA + 1.$ Moreover, $g$ has order at most~$2$. 

       Here, as usual, the spherical derivative is defined by $\displaystyle{g^{\#}(\zeta) = \frac{|g'(\zeta)|}{1 + |g(\zeta)|^2}}.$
    \end{lemma}

    \begin{lemma}\label{3l2}
        Let $k (\ge 1)$ and $l(\ge 0)$ be two integers. Suppose $f$ be a rational function and $\displaystyle{M[f]=\prod_{i=0}^{k}(f^{(i)})^{n_i}}$ be a monomial of degree $d$ and weight $w$. If $f \neq 0$ and $f^{(k)} \neq 0$, then the function $aM[f](z)-z^l$
        has at least one simple zero in the complex plane $\mathbb{C}$, where $a$ is a non-zero constant.
    \end{lemma}
    \begin{proof}
        Since $f\neq 0$ and $f^{(k)}\neq 0$, it follows that $f$ must be a rational function which is not a polynomial. Then $f(z)$ can be written as
        $$f(z)=\frac{C}{(z-z_1)^{r_1}(z-z_2)^{r_2}\cdots (z-z_t)^{r_t}},$$
        where $C(\neq 0)$ is a constant and $r_1,\; r_2,\cdots ,\; r_t$ are all positive integers.\\ Set $\displaystyle{r=\sum_{j=1}^tr_j}$.
        Then we have
        \begin{equation*}
            \begin{split}
                f'(z)= & -C\bigg\{\frac{r_1}{(z-z_1)^{r_1+1}(z-z_2)^{r_2}\cdots (z-z_t)^{r_t}}+\frac{r_2+1}{(z-z_1)^{r_1}(z-z_2)^{r_2+1}\cdots (z-z_t)^{r_t}}+\cdots\\
                +& \frac{r_t+1}{(z-z_1)^{r_1}(z-z_2)^{r_2}\cdots (z-z_t)^{r_t+1}}\bigg\}\\
                =& -\frac{C\displaystyle \sum_{j=1}^{t}r_j(z-z_1)(z-z_2)\cdots(z-z_{j-1})(z-z_{j+1})\cdots(z-z_t)}{(z-z_1)^{r_1+1}(z-z_2)^{r_2+1}\cdots (z-z_t)^{r_t+1}}\\
                =& -\frac{C\left(rz^{t-1}+b_{t-2}z^{t-2}+b_{t-3}z^{t-3}+\cdots+b_1z+b_0 \right)}{\displaystyle \prod_{j=1}^{t}(z-z_j)^{r_j+1}},
            \end{split}
        \end{equation*}
        where $b_0,b_1,\cdots,b_{t-2}$ are constants. Now, by mathematical induction, we get
        \begin{equation*}
            f^{(k)}(z)=\frac{Bz^{k(t-1)}+C_{k(t-1)-1}z^{k(t-1)-1}+\cdots+C_0}{\prod_{j=1}^{t}(z-z_j)^{r_j+k}},
        \end{equation*}
        where $B=(-1)^kr(r+1)\cdots (r+k-1)C\neq 0$ and $C_{k(t-1)-1},\cdots,C_0$ are constants. As $F^{(k)}\neq 0$ and $B\neq 0$, then we have
        $C_{k(t-1)-1}=C_{k(t-1)-2}=\cdots=C_0=0$ and $t-1=0$, which implies $t=1$. Thus, we have
        \begin{equation*}
            f(z)=\frac{C}{(z-z_1)^{r_1}}. 
        \end{equation*}
        Also, for $k\ge 1$, we get 
        \begin{equation*}
            f^{(k)}(z)=\frac{A_k}{(z-z_1)^{r_1+k}};
        \end{equation*}
        and if $k=0$, then $A_0=C$. Therefore, we see that
        \begin{equation*}
            \begin{split}
                M[f](z)= & \left(\frac{A_0}{(z-z_1)^{r_1}}\right)^{n_0}\left(\frac{A_1}{(z-z_1)^{r_1+1}}\right)^{n_1}\cdots \left(\frac{A_k}{(z-z_1)^{r_1+k}}\right)^{n_k}\\
                =& \frac{B_1}{(z-z_1)^{(r_1-1)d+w}}.
            \end{split}
        \end{equation*}
        Now, we have
        \begin{equation}\label{3p1}
            aM[f](z)-z^l=\frac{aB_1-z^l(z-z_1)^{(r_1-1)d+w}}{(z-z_1)^{(r_1-1)d+w}}.
        \end{equation}
        Since the numerator is a polynomial of degree greater than $1$, there exists a point $z_0$ such that $aM[f](z_0)-z_0^l=0$.

        If possible, suppose that all the zeros of $aM[f](z)-z^l$ are multiple zeros.
        We now consider the following two cases.

        \noindent\textbf{Case I:} Suppose $l=0$. 
        
        Then $$M'[f](z)=0\quad \text{ which implies}\quad -\frac{aB_1\left\{(r_1-1)d+w\right\}}{(z-z_0)^{(r_1-1)d+w+1}} =0.$$
        Hence $(r_1-1)d+w=0$, which leads to a contradiction.

   \noindent \textbf{Case II:} Suppose $l\ge 1$.
   
   Then $z_0$ is a root of $aM[f](z)-z^l$. Then we get 
    \begin{equation*}
            aM[f](z_0)-z_0^l=\frac{aB_1-z_0^l(z_0-z_1)^{(r_1-1)d+w}}{(z_0-z_1)^{(r_1-1)d+w}}.
    \end{equation*}
    Then this implies
    \begin{equation}\label{3p2}
        aM'[f](z_0)-lz_0^{l-1}=\frac{-aB_1\{(r_1-1)d+w\}-lz_0^{l-1}(z_0-z_1)^{(r_1-1)d+w+1}}{(z_0-z_1)^{(r_1-1)d+w+1}}=0.
    \end{equation}
    Now, from \eqref{3p1} and \eqref{3p2} we have
    \begin{equation}\label{3p3}
        aB_1=z_0^l(z_0-z_1)^{(r_1-1)d+w}
    \end{equation}
    and
    \begin{equation}\label{3p4}
        aB_1\{(r_1-1)d+w\}=-lz_0^{l-1}(z_0-z_1)^{(r_1-1)d+w+1}.
    \end{equation}
    Together \eqref{3p3} and \eqref{3p4} implies  $\displaystyle{z_0=\frac{lz_1}{(r_1-1)d+w+l}}$. Thus $aM[f](z)-z^l$ has only one zero $z_0$ and $z_1\neq 0$. Hence 
    \begin{equation*}
        aM[f](z)-z^l=\frac{-(z-z_0)^{(r_1-1)d+w+l}}{(z-z_1)^{(r_1-1)d+w}},
    \end{equation*}
    which implies that
    \begin{equation*}
        aB_1-z^l(z-z_1)^{(r_1-1)d+w}=-{(z-z_0)^{(r_1-1)d+w+l}}.
    \end{equation*}
    If $l\ge2$, the coefficient  of $z$ in the left hand side is $0$ and the coefficient of $z$ in right hand side is $$\displaystyle{(-1)^{(r_1-1)d+w+l}\left\{{(r_1-1)d+w+l}\right\}z_0^{{(r_1-1)d+w+l}-1}}\neq0,\quad \text{a contradiction.} $$ \\
    If $l=1$, equating the coefficient of $z$, we have
    $$(-1)^{(r_1-1)d+w+1}\{{(r_1-1)d+w+1}\}z_0^{{(r_1-1)d+w}}=(-1)^{(r_1-1)d+w+1}z_1^{(r_1-1)d+w},$$
    which implies
    $${(r_1-1)d+w+1}\left\{\frac{z_1}{({(r_1-1)d+w+1}}\right\}^{(r_1-1)d+w}=z_1^{(r_1-1)d+w}.$$
    Then we have $\displaystyle{\left\{{(r_1-1)d+w+1}\right\}^{(r_1-1)d+w-1}=1}$, which leads to a contradiction.

    So, \textbf{Case I} and \textbf{Case II} together prove that $aM[f](z)-z^l$ has at least one simple zero.
    \end{proof}
    \begin{lemma}\label{3l3}
        Let $\mathscr{F}$ be a family of meromorphic functions defined in a domain $D$ and let $k$, $n$ be two positive integers. Let $b(z) \not\equiv 0$ and $a_1(z), a_1(z), \ldots, a_{n}(z)$ be analytic functions in $D$.  Suppose $P[f]$ be a homogeneous differential polynomial of degree $d$ and weight $w$ such that  $P[f]$ has only one monomial $M_i[f]$ for some $i\in \{1,2,\cdots,n\}$ with $w$ as its weight, for which the corresponding coefficient $a_i$ is a zero-free holomorphic function.  If, for every function $f \in \mathscr{F}$, $f \neq 0, \quad P[f]\neq 0$  and all zeros of $P[f](z) - b(z)$ have multiplicity at least  $\displaystyle{\frac{w+1}{w-1}}$, then $ \mathscr{F}$ is normal in $D$.
    \end{lemma} 
    \begin{proof}
         Since normality is a local property, for generality, we consider $D=\{z: |z|<1\}$. Suppose $\mathscr{F}$ is not normal at $z_0$. So by Lemma \ref{3l1}, there exist a sequence $\{f_j\}$ of functions in $\mathscr{F}$, a sequence $\{z_j\}$ in $\mathbb{C}$ such that $z_j\longrightarrow z_0$ and a sequence $\{\rho_j\}$ of positive numbers such that $\rho_j\longrightarrow 0$ as $j\longrightarrow \infty$ and 
         \begin{equation}\label{3p5}
             g_j(\zeta)=\frac{f_j(z_j+\rho_j\zeta)}{\rho_j^{\frac{w-d}{d}}}\longrightarrow g(\zeta)
         \end{equation}
         spherically uniformly on a compact subset of $\mathbb{C}$, where $g(\zeta)$ is a non-constant meromorphic function on $\mathbb{C}$. Now Hurwitz's theorem gives $g(\zeta)\neq 0$.
         Then from \eqref{3p5}, we have
         \begin{equation*}
              g_j^{(i)}(\zeta)=\frac{\rho_j^{i}f_j^{(i)}(z_j+\rho_j\zeta)}{\rho_j^{\frac{w-d}{d}}};\quad i=1,2,\cdots,k.
         \end{equation*}
         Therefore
         \begin{equation*}
             M_q[g_j](\zeta)=\frac{\rho_j^{w_q-d}M_q[f_j](z_j+\rho_j\zeta)}{\rho_j^{w-d}};\quad q=1,2,\cdots,n.
         \end{equation*}
         Without loss of generality, we assume that $w$ occurs in $M_1$. Then $a_1$ is a zero-free holomorphic function.
         So, 
         \begin{equation}\label{3p6}
             P[f_j](z_j+\rho_j\zeta)=\sum_{q=1}^n\rho_j^{w-w_q}a_q(z_j+\rho_j\zeta)M_q[g_j](\zeta)\longrightarrow a_1(z_0)M_1[g](\zeta),
         \end{equation}
         which implies
         \begin{align}\label{3p7}
             P[f_j](z_j+\rho_j\zeta)-b(z_j+\rho_j\zeta) = & a_1(z_j+\rho_j\zeta)M_1[g_j](\zeta)+\sum_{q=2}^n\rho_j^{w-w_q}a_q(z_j+\rho_j\zeta)M_q[g_j](\zeta)\nonumber\\
             &-b(z_j+\rho_j\zeta)\nonumber\\
             &\longrightarrow a_1(z_0)M_1[g](\zeta)-b(z_0) \quad \text{as}\quad j\longrightarrow\infty.
         \end{align}
         As $P[f_j](z_j+\rho_j\zeta)\neq 0$, by Hurwitz's theorem, we may say either $M_1[g](\zeta)\equiv 0$ or $M_1[g](\zeta)\not\equiv 0$ for any $\zeta\in \mathbb{C}$ that is not a pole of $g(\zeta)$. If $M_1[g](\zeta)\equiv 0$, reduce that $g$ is a non-zero constant since $g\neq 0$. Thus $M_1[g](\zeta)\neq 0$.

         From \eqref{3p7}, by Hurwitz's theorem, we have that all zeros of $a_1(z_0)M_1[g](\zeta)-b(z_0)$ have multiplicities at least $\displaystyle{\frac{w+1}{w-1}}$.

         If $M_1[g]$ is a rational function, then $g$ is a rational function. Since $g\neq 0$ and $M_1[g]\neq0$, then by Lemma \ref{3l2} $a_1(z_0)M_1[g]-b(z_0)$ has at least one simple zero, which contradicts the fact that all zeros of $a_1(z_0)M_1[g]-b(z_0)$ have multiplicites at least $\displaystyle{\frac{w+1}{w-1}}\; (>1)$. Thus, $M_1[g]$ is transcendental, so by Nevanlinna's 2nd fundamental theorem, we have
         \begin{equation*}
             \begin{split}
                 &T(r,M_1[g])\\
                 \le & \overline{N}\left(r,a_1(z_0)M_1[g]\right)+\overline{N}\left(r,\frac{1}{a_1(z_0)M_1[g]}\right)+\overline{N}\left(r,\frac{1}{a_1(z_0)M_1[g]-b(z_0)}\right)+S\left(r, M_1[g]\right)\\
                 \le & \frac{1}{w}N\left(r,a_1(z_0)M_1[g]\right)+\frac{w-1}{w+1}N\left(r,a_1(z_0)M_1[g]\right)+S\left(r, M_1[g]\right)\\
                 \le &\frac{w^2+1}{w^2+w}T\left(r,M_1[g]\right)+ S\left(r, M_1[g]\right),
             \end{split}
         \end{equation*}
         which implies that $T\left(r,aM_1[g]\right)=S\left(r,M_1[g]\right)$, a contradiction. This completes the proof.
    \end{proof}
        \section{\textbf{Proof the theorem \ref{m}}}
    \begin{proof}
        Without loss of generality, we consider $w$ occurs in $M_1$ and hence $a_1$ is zero-free.

        It is known that normality is a local property, then for our convenience, assume that $D=\Delta= \{z:\; |z|<1\}$, and $\psi(z)=z^l\varphi(z)\; (z\in \Delta)$, where $l$ is a positive integer, $\varphi(0)=1,\; \varphi(z)\neq 0$ on $\Delta'=\{z:\; 0<|z|<1\}$. By the Theorem \ref{c}, $\mathscr{F}$ is normal in $\Delta'=\{z:\; 0<|z|<1\}$, so it is only required to establish that $\mathscr{F}$ is normal at $z=0$.

        Now, we consider $\displaystyle{\mathscr{G}=\left\{ g(z)=\frac{f(z)}{\psi(z)}:\; f\in\mathscr{F}, z\in \Delta \right\}}$, where for all $g\in\mathscr{G},\quad g(0)=\infty$, as $f\neq0, \quad f\in \mathscr{F} $. It is required to prove that $\mathscr{G}$ is normal in $\Delta$. 
        If possible, suppose that $\mathscr{G}$ is not normal at $z_0\in \delta$. Then, by Lemma \ref{3l1}, there exist a sequence of functions $\{g_j\}$ in $\mathscr{G}$, a sequence $\{z_j\}$ in $\mathbb{C}$ such that $z_j\longrightarrow z_0$ and a sequence $\{\rho_j\}$ of positive numbers such that $\rho_j\longrightarrow 0$ as $j\longrightarrow \infty$ and 
         \begin{equation*}
             G_j(\zeta)=\frac{g_j(z_j+\rho_j\zeta)}{\rho_j^{\frac{w-d}{d}}}\longrightarrow G(\zeta) \quad \left(\text{as} \quad j\longrightarrow \infty\right),
         \end{equation*}
        spherically uniformly on a compact subset of $\mathbb{C}$, where $G(\zeta)$ is a non-constant meromorphic function on $\mathbb{C}$ and $G(\zeta)\neq 0$.
         We now consider the following two cases.
         
         \noindent\textbf{Case I:} Suppose $\displaystyle{\frac{z_j}{\rho_j}\longrightarrow \infty \quad \text{as}\quad j\longrightarrow\infty}$.
         
         Then, using Leibniz's rule of derivative on $f_j(z)=\psi(z)g_j(z)$, we have
         \begin{align*}
             & f_j^{(i)}(z) \nonumber\\ 
             = & {i\choose 0}g_j^{(i)}(z)\psi(z)+{i\choose 1}g_j^{(i-1)}(z)\psi'(z)+\cdots +{i\choose p}g_j^{(i-p)}(z)\psi^{(p)}(z)+\cdots+{i\choose i}g_j(z)\psi^{(i)}(z)\\
             =& g_j^{(i)}(z)\psi(z)+\sum_{p=1}^{i}g_j^{(i-p)}(z)\psi^{(p)}(z),
         \end{align*}
         which implies that
         \begin{equation}\label{3p8}
             g_j^{(i)}(z)=\frac{f_j^{(i)}(z)}{\psi(z)}-\sum_{p=1}^{i}{i\choose p}g_j^{(i-p)}(z)\frac{\psi^{(p)}(z)}{\psi(z)}.
         \end{equation}
         Again
         \begin{align}
             \psi^{(p)}(z)= &\sum_{t=0}^pl(l-1)(l-2)\cdots(l-p+t+1){p\choose t}z^{l-p+t}\varphi^{(t)(z)}\nonumber\\ 
             =& \sum_{t=0}^pA_{pt}z^{l-p+t}\varphi^{(t)}(z),
         \end{align}
         where 
         \begin{align*}
             A_{pt}=& l(l-1)(l-2)\cdots(l-p+t+1){p\choose t},\quad \text{if}\quad l\ge p\\
             =& 0; \quad \text{if}\quad l<p \quad\text{and}\quad t=0,1,\cdots,p-l-1\\
            = & 1; \quad \text{if}\quad t=p.
         \end{align*}
         Thus from \eqref{3p8} we have 
         \begin{equation}\label{3p9}
             g_j^{(i)}(z)=\frac{f_j^{(i)}(z)}{\psi(z)}-\sum_{p=1}^{i}{i\choose p}g_j^{(i-p)}(z)\sum_{t-0}^{p}A_{pt}\frac{1}{z^{p-t}}\frac{\varphi^{(t)}(z)}{\varphi(z)},\quad i=0,2,\cdots,k.
         \end{equation}
         Since
         \begin{equation*}
             G_j^{(i)}=\frac{\rho_j^{i}g_j^{(i)}(z_j+\rho_j\zeta)}{\rho_j^\frac{w-d}{d}};\quad i=0,1,2,\cdots,k,
         \end{equation*}
         we have 
         \begin{equation*}
            \left(\rho_j^\frac{w-d}{d}\right)^{n_{qi}} \left(G_j^{(i)}\right)^{n_{qi}}=\left(\rho_j^{i}\right)^{n_{qi}}\left({g_j^{(i)}(z_j+\rho_j\zeta)}\right)^{n_{qi}};\quad i=0,1,2,\cdots,k;\quad q=1,2,\cdots,n.
         \end{equation*}
         Therefore 
         \begin{equation}\label{3p10}
             \rho_j^{w-w_q}M_q[G_j](\zeta)=M_q[g_j](z_j+\rho_j\zeta).
         \end{equation}
         Now, from \eqref{3p9}, we get
         \begin{align}\label{3p11}
             & M_q[g_j](z_j+\rho_j\zeta)\nonumber\\
             =&\prod_{i=0}^{k}\left[\frac{f_j^{(i)}(z_j+\rho_j\zeta)}{\psi(z_j+\rho_j\zeta)}-\sum_{p=1}^{i}\left\{g_j^{(i-p)}{(z_j+\rho_j\zeta)\sum_{t=0}^{p}A_{pt}\frac{1}{(z_j+\rho_j\zeta)^{p-t}}\frac{\varphi^{(t)}(z_j+\rho_j\zeta)}{\varphi(z_j+\rho_j\zeta)}}\right\}\right]^{n_{qi}}.
         \end{align}
         Now, from \eqref{3p10} and \eqref{3p11} we get
         \begin{align}\label{3p12}
             & \rho_j^{w-w_q}M_q[G_j](\zeta)\nonumber\\=& \prod_{i=0}^{k}\left[\frac{f_j^{(i)}(z_j+\rho_j\zeta)}{\psi(z_j+\rho_j\zeta)}-\sum_{p=1}^{i}\left\{g_j^{(i-p)}{(z_j+\rho_j\zeta)\sum_{t=0}^{p}A_{pt}\frac{1}{(z_j+\rho_j\zeta)^{p-t}}\frac{\varphi^{(t)}(z_j+\rho_j\zeta)}{\varphi(z_j+\rho_j\zeta)}}\right\}\right]^{n_{qi}}\nonumber\\
             =&  \prod_{i=0}^{k}\left[\frac{f_j^{(i)}(z_j+\rho_j\zeta)}{\psi(z_j+\rho_j\zeta)}-\sum_{p=1}^{i}\left\{\frac{g_j^{(i-p)}(z_j+\rho_j\zeta)}{\rho_j^p}\sum_{t=0}^{p}\frac{A_{pt}}{\left(\frac{z_j}{\rho_j}+\zeta\right)^{p-t}}\frac{\rho_j^{t}\varphi^{(t)}(z_j+\rho_j\zeta)}{\varphi(z_j+\rho_j\zeta)}\right\}\right]^{n_{qi}}.
         \end{align}
         It is clear that 
         \begin{align*}
             \lim_{j\longrightarrow\infty}\frac{1}{\frac{z_j}{\rho_j}+\zeta}=0\quad\text{and}\quad \lim_{j\longrightarrow\infty}\frac{\rho_j^t\varphi^t(z_j+\rho_j\zeta)}{\varphi(z_j+\rho_j\zeta)}=0, \quad \text{for}\quad t\ge 1.
         \end{align*}
         Here $\displaystyle{\frac{g_j^{(i-p)}(z_j+\rho_j\zeta)}{\rho_j^p}}$ is locally bounded on $\mathbb{C}\backslash$ $\{\text{all poles of}\; G(\zeta)\}$, as  $\displaystyle{\frac{g_j(z_j+\rho_j\zeta)}{\rho_j^{\frac{w-d}{d}}}\longrightarrow G(\zeta)}$. 
         Now, \eqref{3p11}  can be written as 
         
         \begin{align}\label{3p13}
             M_q[g_j](z_j+\rho_j\zeta)=& \frac{M_q[f_j](z_j+\rho_j\zeta)}{\psi^{d}(z_j+\rho_j\zeta)}+R(g_j(z_j+\rho_j\zeta)),\psi_j(z_j+\rho_j\zeta),z_j+\rho_j\zeta)\nonumber\\
             \text{or,}\quad \frac{M_q[f_j](z_j+\rho_j\zeta)}{\psi^{d}(z_j+\rho_j\zeta)}=& M_q[g_j](z_j+\rho_j\zeta)-R(g_j(z_j+\rho_j\zeta)),\psi_j(z_j+\rho_j\zeta),z_j+\rho_j\zeta).
         \end{align}
         Thus, for $q=1$, together \eqref{3p9}, \eqref{3p10}, \eqref{3p12} and \eqref{3p13}  on every compact subset of $\mathbb{C}$, which contains no pole of $G(\zeta)$, we have
         \begin{align}\label{3p14}
             \frac{M_1[f_j](z_j+\rho_j\zeta)}{\psi^{d}(z_j+\rho_j\zeta)}\longrightarrow M_1[g](\zeta) \quad \text{as}\quad j\longrightarrow\infty.
         \end{align}
         If $q\neq1$, then 
         \begin{align*}
             \frac{M_q[f_j](z_j+\rho_j\zeta)}{\psi^{d}(z_j+\rho_j\zeta)}\longrightarrow 0\quad \text{as}\quad j\longrightarrow\infty.
         \end{align*}
         So,
         \begin{align}\label{3p15}
         \frac{\displaystyle{\sum_{q=1}^{n}a_q(z_j+\rho_j\zeta)M_q[f_j](z_j+\rho_j\zeta)}}{\psi^{d}(z_j+\rho_j\zeta)}\longrightarrow a_1(z_0)M_1[G](\zeta)\quad {as}\quad j\longrightarrow \infty
         \end{align}
         and
         \begin{align*}
             \frac{\displaystyle{\sum_{q=1}^{n}a_q(z_j+\rho_j\zeta)M_q[f_j](z_j+\rho_j\zeta)}-\psi^{d}(z_j+\rho_j\zeta)}{\psi^{d}(z_j+\rho_j\zeta)}\longrightarrow a_1(z_0)M_1[G](\zeta)-1\quad \text{as}\quad j\longrightarrow\infty.
         \end{align*}
         Since $\displaystyle{\sum_{q=1}^{n}a_q(z_j+\rho_j\zeta)M_q[f_j](z_j+\rho_j\zeta)}\neq 0$, then we may say that $$ \displaystyle{\sum_{q=1}^{n}a_q(z_j+\rho_j\zeta)M_q[f_j](z_j+\rho_j\zeta)}-\psi^{d}(z_j+\rho_j\zeta)\quad  \text{and} \quad\psi(z_j+\rho_j\zeta)$$ have no common zeros. Next, we can arrive at a contradiction by using the same argument as in the latter part of Lemma \ref{3l3}.

         \noindent\textbf{Case II:} Suppose $\frac{z_j}{\rho_j}\longrightarrow \alpha \quad \text{a finite number}$ as $j\longrightarrow\infty$. 
         
         Then
         \begin{align*}
             \frac{g_j(\rho_j\zeta)}{\rho_j^{\frac{w-d}{d}}}= \frac{g_j(z_j+\rho_j(\zeta-\frac{z_j}{\rho_j}))}{\rho_j^{\frac{w-d}{d}}}=G_j(\zeta-\frac{z_j}{\rho_j})\longrightarrow G(\zeta-\alpha)=G_1(\zeta)\quad \text{(say)}. 
         \end{align*}
         Clearly $G_1(\zeta)\neq 0$ and $\zeta=0$ is a pole of $G_1$ with order at least $l$.
         Now, consider
         \begin{align}\label{3p16}
             H_j(\zeta)=\frac{f_j(\rho_j\zeta)}{\rho_j^{\frac{w-d}{d}+l}}.
         \end{align}
         So, \begin{align}\label{3p17}
             H_j(\zeta)=\left(\frac{\psi(\rho_j\zeta)}{\rho_j^{l}}\right)\left(\frac{f_j(\rho_j\zeta)}{\rho_j^{\frac{w-d}{d}}\psi(\rho_j\zeta)}\right)=\left(\frac{\psi(\rho_j\zeta)}{\rho_j^{l}}\right)\left(\frac{g_j(\rho_j\zeta)}{\rho_j^{\frac{w-d}{d}}}\right).
         \end{align}
         It is clear that $\frac{\psi(\rho_j\zeta)}{\rho_j^{l}}=\zeta^l\varphi(\rho_j\zeta)\longrightarrow \zeta^l$ as $j\longrightarrow \infty$. Thus
         \begin{align*}
             H_j(\zeta)\longrightarrow \zeta^lG(\zeta-\alpha)=\zeta^lG_1(\zeta) \quad \left(\text{as}\quad j\longrightarrow \infty\right),
         \end{align*}
         uniformly on compact subset of $\mathbb{C}$. Since $G_1$ has a pole of order at least $l$ at $\zeta=0$, we have $H(0)\neq0\implies H(\zeta)\neq 0$.
         Now, from \eqref{3p16} we get
         \begin{align*}
             H_j^{(i)}(\zeta)=\frac{f_j^{(i)}(\rho_j\zeta)}{\rho_j^{\frac{w-d}{d}+l-i}}\longrightarrow H^{(i)}(\zeta) \quad \left(\text{as}\quad j\longrightarrow\infty\right)
         \end{align*}
         spherically uniformly on compact subset of $\mathbb{C}\backslash \{\text{all the poles of}\; G_1(\zeta)\}$. So on every compact subsets of $\mathbb{C}$ which contains no pole of $G(\zeta)$, we have 
         \begin{align*}
             \rho_j^{w-d}M_q[H_j](\zeta)=\rho_j^{w_q-d}\frac{M_q[f_j](\rho_j\zeta)}{\rho_j^{ld}}
         \end{align*}
         and when $q=1$, then $w_1=w$. So 
         \begin{align*}
             \frac{M_1[f_j](\rho_j\zeta)}{\rho_j^{ld}}\longrightarrow M_1[H](\zeta) \quad \left(\text{as}\quad j\longrightarrow \infty\right)
         \end{align*}
         spherically uniformly on a compact subset of $\mathbb{C}\backslash$ $\{\text{all poles of}\; G_1(\zeta)\}$, and 
         \begin{align*}
              \rho_j^{w-w_q}M_q[H_j](\zeta)=\frac{M_q[f_j](\rho_j\zeta)}{\rho_j^{ld}}\longrightarrow 0 \quad \text{as} \quad j\longrightarrow \infty.
         \end{align*}
         So, on every compact subset of $\mathbb{C}$, which contains no poles of $G(\zeta)$, we have
    \begin{align}\label{3p18}
    \frac{\displaystyle{\sum_{q=1}^{n}a_q(\rho_j\zeta)M_q[f_j](\rho_j\zeta)}}{\rho_j^{ld}}\longrightarrow a_1(z_0)M_1[H](\zeta)\quad \left(\text{as}\quad j\longrightarrow \infty\right)
    \end{align}
    and 
    \begin{align}\label{3p19}
    \frac{\displaystyle{\sum_{q=1}^{n}a_q(\rho_j\zeta)M_q[f_j](\rho_j\zeta)}-\psi^{d}(\rho_j\zeta)}{\rho_j^{ld}}\longrightarrow a_1(z_0)M_1[H](\zeta)-\zeta^{ld}\quad \left(\text{as}\quad j\longrightarrow \infty\right)
    \end{align}
    locally uniformly on $\mathbb{C}$. By the assumption of the theorem and \eqref{3p19}, Hurwitz theorem gives that, all the zeros of $a_1(z_0)M_1[H](\zeta)-\zeta^{ld}$ have multiplicities at least $\displaystyle{\frac{w+1}{w-1}}$. 
    
    Similarly, we see from \eqref{3p17}, that either $\displaystyle{a_1(z_0)M_1[H](\zeta)\neq 0}$ or, $a_1(z_0)M_1[H](\zeta)\equiv 0$.

    If $M_1[H](\zeta)\equiv 0$, then  $\displaystyle{a_1(z_0)M_1[H](\zeta)-\zeta^{ld}=-\zeta^{ld}}$ has only one zero with multiplicity $ld$.
    
    Now, let $w=2$. By the hypothesis of the Theorem, we have  $ld\le 2$. Hence, the function $\displaystyle{a_1(z_0)M_1[H](\zeta)-\zeta^{ld}}$ has only one zero with multiplicity $\le2$, which contradicts the fact that all zeros of this function have multiplicities $\displaystyle{\frac{w+1}{w-1}=3}$.

    If $w\ge 3$, then $ld=1$. Then by a similar argument we can arrive at a contradiction, since $\displaystyle{\frac{w+1}{w-1}>1}$. Hence $M_1[H](\zeta)\neq 0$. 

    Assume that, $M_1[H](\zeta)$ is rational. Noting that, $M_1[H](\zeta)\neq 0$ and by Lemma \ref{3l3}, $a_1(z_0)M_1[H](\zeta)-\zeta^{ld}$ has at least one simple zero, which leads to a contradiction, since all zeros of $a_1(z_0)M_1[H](\zeta)-\zeta^{ld}$ have multiplicities at least $\displaystyle{\frac{w+1}{w-1}}$. So, $M_1[H](\zeta)$ is a transcendental function.

    So by the First and Second fundamental Theorems of Nevanlinna, we have
    \begin{align*}
        T\left(r,a_1(z_0)M_1[H]-\zeta^{ld}\right)=& T\left(r,M_1[H]\right)+S(r, M_1[H])
    \end{align*}
    \begin{align}\label{3p20}
        \& \quad T\left(r,M_1[H]\right)\le  & \overline{N}\left(r,\frac{1}{M_1[H]}\right)+\overline{N}\left(r,\frac{1}{a_1(z_0)M_1[H]-\zeta^{ld}}\right)+\overline{N}\left(r,{M_1[H]}\right)\nonumber\\
        &+S(r,M_1[H]).
    \end{align}
   
    Since $\displaystyle{\overline{N}(r,M_1[H])\le   \frac{1}{w}N(r,M_1[H])}$, then from \eqref{3p20} we get
    \begin{align*}
        T(r,M_1[H])\le & \frac{1}{w}N(r,M_1[H])+\frac{w-1}{w+1}N\left(r,\frac{1}{a_1(z_0)M_1[H]-\zeta^{ld}}\right)+S(r,M_1[H])\\
        \le & \frac{1}{w}N(r,M_1[H])+\frac{w-1}{w+1}N(r,M_1[H])+S(r,M_1[H])\\
        \le & \frac{w^2+1}{w^2+w}N(r,M_1[H])+S(r,M_1[H])\\
        \le & \frac{w^2+1}{w^2+w}T(r,M_1[H])+S(r,M_1[H]).
    \end{align*}
    Since $\displaystyle{\frac{w^2+1}{w^2+w}<1}$, then $T(r,M_1[H])=S(r,M_1[H])$, which leads to a contradiction. So, $\mathscr{G}$ is normal in $\Delta$. 

    It remains to show that the family $\mathscr{F}$ is normal at $z = 0$. Since $\mathscr{G}$ is normal in $\Delta$, it is equicontinuous with respect to the spherical distance. Moreover, as  $g(0) = \infty$ for each $g \in \mathscr{G}$, there exists $\delta > 0$ such that $|g(z)| > 1$ for every $g \in \mathscr{G}$ and each $z \in \Delta_\delta = \{ z: |z| < \delta \}$. 

    Suppose, to the contrary, that the family $\mathscr{F}$ is not normal at $z=0$. Now, consider the family 
    \begin{align*}
        \mathscr{F}_1 = \left\{ \frac{1}{f}: f \in \mathscr{F} \right\}.
    \end{align*}
    Then, it is clear that $ \mathscr{F}_1$ is normal in $\Delta' = \{ z: 0 < |z| < 1 \}$, but not at $z = 0$. Hence there exists a sequence $\left\{ \frac{1}{f_j} \right\} \subset \mathscr{F}_1$ that converges locally uniformly in $\Delta'$ but fails to converge in $\Delta$. 

    Since $f_j \neq 0$ in $\Delta$, each $\frac{1}{f_j}$ is holomorphic in $\Delta$. According to  the maximum modulus principle we have $\frac{1}{f_j} \to \infty$ in $\Delta'$, which implies that $f_j \to 0$ locally and uniformly in $\Delta'$; so does $\{ g_j \} \subset \mathscr{G}$, where $g_j = \frac{f_j}{\varphi}$. However, this contradict the fact that $|g_j(z)| \ge 1$ for each $z \in \Delta_\delta$. 

This concludes the proof of Theorem \ref{m}.
    \end{proof}
	\bibliographystyle{acm}
	\bibliography{mybib.bib}
\end{document}